\documentclass[11pt,twoside]{amsart}
\usepackage{mathtools}
\usepackage{amssymb}
\usepackage{comment}
\usepackage{enumitem}
\usepackage{tikz-cd}

\usepackage{color}
\usepackage{hyperref}
\usepackage[alphabetic,backrefs]{amsrefs}

\usepackage[T1]{fontenc} 
\usepackage[utf8]{inputenc}
\usepackage{cancel}
\usepackage[normalem]{ulem}

\theoremstyle{plain}
\newtheorem{lem}{Lemma}[section]

\newtheorem{defn}[lem]{Definition}
\newtheorem{thm}[lem]{Theorem}
\newtheorem{prop}[lem]{Proposition}

\newcommand{\R}{\mathbb{R}}

\newcommand{\Z}{\mathbb{Z}}
\newcommand{\C}{\mathbb{C}}

\newcommand{\loc}{\text{loc}}
\newcommand{\Morse}{\operatorname{Morse}}

\newcommand{\Crit}{\operatorname{Crit}}
\newcommand{\A}{\mathcal A}
\newcommand{\E}{\mathcal E}
\newcommand{\Intervals}{\mathcal I}

\newcommand{\HZ}{\text{HZ}}

\usepackage[backgroundcolor=white,obeyDraft]{todonotes}

\definecolor{lightcyan}{rgb}{0.0 1.0 1.0}

\newcommand{\sam}[1]
{\protect\todo[inline,color=yellow,author=Sam]{#1}}

\newcommand{\tony}[1]
{\protect\todo[inline,color=lightcyan,author=Tony]{#1}}

\DeclarePairedDelimiter{\floor}{\lfloor}{\rfloor}

\newcommand{\tonycut}[1]{}

\newcommand{\defin}[1]{\textbf{#1}}

\title[Lagrangian H-Z capacities and energy-capacity inequalities]{Lagrangian Hofer-Zehnder capacities and energy-capacity 
	inequalities}
\author{Samuel Lisi}
\author{Antonio Rieser}

\thanks{Samuel Lisi was partially supported by NSF award \# 1929176 and by the ERC  
	Starting Grant of Vincent Colin (Geodycon). Antonio Rieser was partially supported by CONAHCYT/SECIHTI Investigadoras y Investigadores
	por México Project 1076, CONAHCYT Ciencia de Frontera Grant CF-2019-217392, ERC
	2012 Advanced Grant 20120216, ISF grant 723/10, and Short Visit Grants from the
	Contact and Symplectic Topology (CAST) Research Network of the European Science
	Foundation.}

\begin{document}

	\begin{abstract}
		We introduce a new coisotropic Hofer-Zehnder capacity and use it to prove an energy-capacity inequality for displaceable Lagrangians.
	\end{abstract}
\maketitle

\section{Introduction and main results}

In this paper, we revisit the coisotropic capacity introduced in
\cite{Lisi_Rieser_2020}, but consider the special case when the 
coisotropic submanifold in question is a Lagrangian
submanifold. In this special case, we establish 
a version 
of an energy-capacity
inequality relating a coisotropic capacity to displacement energy,
specifically in the case of a monotone Lagrangian.

There have been a number of interesting papers that have built on  ideas
connected to coisotropic capacities \citelist{
\cite{Shi_2024}
\cite{Shi_2024_b}
\cite{Jin_Lu_2020}
\cite{Jin_Lu_2023}
\cite{Jin_Lu_2023_b}
}. 
To the best of our knowledge, all of these papers (and of course
\cite{Lisi_Rieser_2020}) feature some form of variational principle,
and are thus focused on problems in which the coisotropic submanifold
is a linear subspace of $\R^{2n}$ or similar. 
The principal new element of
our current paper is that we work with Floer homology and consider
arbitrary monotone symplectic manifolds. The price we pay, however,
is that we are restricted to cases where Floer homology is defined.
Here, we treat the case of monotone Lagrangian
submanifolds.
Our perspective here can be thought of as closely connected to the ideas 
of Lagrangian spectral invariants \citelist{\cite{Leclercq_2008}
    \cite{Leclercq_Zapolsky_2018} 
    \cite{Kislev_Sheluskhin_2022}
    \cite{Kawasaki_2018}
}. A different approach to spectral capacities related to Lagrangian submanifolds was taken in \cite{Cant_Zhang_2024}, in which they investigated spectral capacities of neighborhoods of Langrangian submanifolds.

Our 
energy-capacity inequality relates {the} displacement energy of the
Lagrangian and a modified version of the coisotropic Hofer-Zehnder capacity 
{from \cite{Lisi_Rieser_2020}}.
We study a modified version of that coiostropic capacity here, where the class
of admissible Hamiltonians {is restricted} to those that achieve their maximum on a ball.    
(The precise formulation is given in Definition \ref{def:modified_admissible}.)
Heuristically speaking, this means that the
maximum of $H$ represents the unit in the Lagrangian Quantum ring and, in the
case that the Lagrangian is displaceable and is monotone with minimal Maslov
number at least 2, the
chords/critical points that kill it must have $H=0$ with a connecting
Floer trajectory that represents a non-trivial relative homology class.
The monotonicity condition on $L, (W, \omega)$ is also necessary so that the Quantum
Lagrangian Complex described by Biran and Cornea
\citelist{\cite{Biran_Cornea_new_persp} \cite{Biran_Cornea_pp_2007}} is well-defined.
{Given this coisotropic capacity, our first theorem is the following.}
\begin{thm}
    
    \label{T:displacement}
    Let $(W,\omega)$ be a compact or geometrically bounded symplectic manifold, and suppose that $L \subset (W, \omega)$ is a connected, closed Lagrangian. 
    Suppose, furthermore, that $L$ is monotone with
    minimal Maslov number at least $2$. 

    Let $d(L)$ denote the displacement energy of the Lagrangian $L$ and let 
    $\hat c_\HZ(W, L)$ denote the modified relative Hofer-Zehnder capacity. Then 
    \[
        \hat c_\HZ(W, L) \le d(L).
    \]
\end{thm}
This is proved in Section \ref{S:displaceableLag}. The key technical ingredients
of the proof are the use of persistent homology for filtered Floer homology
(following ideas from Polterovich and Shelukhin
\cite{Polterovich_Shelukhin_2016}{Definitions 2.6, 2.7 and Remark 2.9} and
local Floer homology \citelist{\cite{Ginzburg_2010} \cite{Ginzburg_Gurel_2010}}.
As a corollary, this inequality gives a new proof of the bound on the relative
Gromov radius of Barraud-Cornea \cite{Barraud_Cornea_2007}, 
obtained by Biran and Cornea \cite{Biran_Cornea_pp_2007}*{Proposition 6.6.1}.

In Section \ref{sec:counterexample}, we illustrate the necessity of the
modification of the relative Hofer-Zehnder capacity.  We construct a family 
of displaceable Lagrangians, which have uniformly bounded displacement energy, but
whose relative Hofer-Zehnder capacity, using the
definition we gave in \cite{Lisi_Rieser_2020}, can be made to be arbitrarily large. We note that a different energy-capacity inequality for the original relative Hofer-Zehnder capacity is found in \cite{Humiliere_Leclercq_Seyfaddini_2013} (where it is also attributed to us), and proved via different methods in \cite{Borman_McLean_2014}

\subsection*{Acknowledgements}

We would like to thank 
Vincent Humili\`ere, R\'emi Leclercq and Sobhan Seyfaddini 
for helpful conversations and encouragement.

\section{The modified coiostropic Hofer-Zehnder capacity}

\subsection{The modified relative Hofer-Zehnder capacity}

In the following, let $(W, \omega)$ be a symplectic manifold 
and let $L \subset W$ be an embedded Lagrangian
submanifold. We will later impose additional conditions on {both} $(W, \omega)$ and
$L$. We define the modified relative Hofer-Zehnder capacity of $L$ as follows. 

\begin{defn} \label{def:modified_admissible} 
An autonomous Hamiltonian $H \colon W \to \R$ is admissible if all of the following properties
hold:
\begin{enumerate}
    \item All Hamiltonian trajectories $x(t)$ such that $x(0) \in L$ and $x(T)
         \in L$ for some $T > 0$ must either have $T > 1$ or are constant
         trajectories. 
    \item The Hamiltonian $H$ achieves its maximum $m(H)$ on a set that
        intersects $L$. Furthermore, this level set on $L$ is diffeomorphic to
        a (closed) $n$-dimensional ball.
     \item \label{item:Condition 3} The Hamiltonian $H$ restricted to $L$ has two critical values, $0$
         and $m(H)$. 
     \item The Hamiltonian $H$ is compactly supported in $W$. (Its $0$ set
         must also intersect $L$ by the condition (\ref{item:Condition 3}).)
\end{enumerate}
Note, in particular, that if $H$ is admissible then so is $\tau H$ for any $\tau
\in [0, 1]$. 

The (modified) relative Hofer-Zehnder capacity $\hat c_{\HZ}(W, L)$ is defined to be
\[
    \hat c_{\HZ}(W, L) = \sup \{ m(H) \, | \, H \text{ is admissible.} \}.
\]
\end{defn}

Notice that there are two new features of the modified relative Hofer-Zehnder capacity
as compared to the relative Hofer-Zehnder capacity we introduced in 
\cite{Lisi_Rieser_2020}: the first is cosmetic, and involves
replacing $H$ by $m(H) - H$; the second is essential 
and regards the topology of the set on
which $H$ achieves its maximum, which we now require to be diffeomorphic to a closed $n$-dimensional ball. The first change is for ease of compatibility
with the sign conventions of Biran and Cornea
\citelist{\cite{Biran_Cornea_pp_2007} \cite{Biran_Cornea_new_persp}}. The
second change is essential for the inequality we will prove. Its necessity is
illustrated in Section \ref{sec:counterexample}.

Observe that despite these changes, the monotonicity and conformality axioms
of a coisotropic capacity, as defined in \cite{Lisi_Rieser_2020}, are satisfied. 
The lower bound on the original coisotropic capacity
of a ball in \cite{Lisi_Rieser_2020} was obtained by a construction using a radial Hamiltonian. This same radial Hamiltonian is also admissible in the
more restrictive sense of Definition \ref{def:modified_admissible}, and so the same lower bound is also satisfied. 
The upper bound also holds by the elementary observation that the modified
relative Hofer-Zehnder capacity is bounded above by the un-modified relative
Hofer-Zehnder capacity. We have therefore verified that $\hat c_{\HZ}$ is an
example of a coisotropic capacity as defined in \cite{Lisi_Rieser_2020}. 

\subsection{Failure of Theorem \ref{T:displacement} for the unmodified capacity} \label{sec:counterexample}

We now illustrate with an example why it was necessary to introduce the modified coisotropic
Hofer-Zehnder capacity. Recall that our original version of the coisotropic
Hofer-Zehnder capacity (from \cite{Lisi_Rieser_2020}) puts no condition on the
topology of the sets on which the admissible Hamiltonian achieves its maximum and
minimum. In the following, we write that a Hamiltonian $H$ is \emph{HZ-admissible} 
if it is admissible in this weaker sense. 

First note the following immediate lemma:
\begin{lem}\label{lem:Product ineq}
Let $(X, \omega_X)$ and $(P, \omega_P)$ be symplectic manifolds. Let $L \subset X$ be a properly embedded Lagrangian and $Q \subset P$ be a closed, embedded Lagrangian. Then $$c(L, X) \le c( L \times Q, X \times P)$$.
\end{lem}
\begin{proof}

Let $H \colon X \to \R$ be an HZ-admissible Hamiltonian with respect to $(X, L)$, 
and let $m(H) = || H || = \max H$. 
Let $\beta \colon P \to [0,1]$ be a cut-off function with 
$\beta \equiv 1$ in a neighbourhood of $Q$ and with compact support in $P$.

Let $\pi_X, \pi_P \colon X \times P$ denote the projections to $X$ and $P$ respectively, and by slight abuse of notation 
write $H= \pi_X^*H$ and $\beta = \pi_P^*\beta$.

Consider now the function $\tilde H \colon X \times P \to \R$ by \[
\tilde H = ( H - m(H) )\beta  + m(H).
\]

Observe that $\tilde H = H$ in a neighbourhood of $L \times Q \subset X \times
P$. The Lagrangian chords of $\tilde H$ relative to $L \times Q$ therefore 
take the form $t \mapsto (x(t), p)$ where $x$ is a chord of $H$ relative to $L$
and $p \in Q$.
It therefore follows that $\tilde H$ is HZ-admissible with respect to $(L \times Q, X \times P)$. 
Furthermore, $|| \tilde H || = m(H)$ and the result now follows.
\end{proof}

From this, we obtain the following
\begin{lem}
    For $\epsilon > 0$, let $D^2(\epsilon)$ denote the disk in $\mathbb{R}^2$ of area of $\epsilon$ and let $S_\epsilon$ be its boundary.  

For each $\epsilon > 0$, let \[
    L_\epsilon = S_\epsilon \times Q \subset D^2( 3 \epsilon ) \times T^*Q.
\]

Then, $L_\epsilon$ is a monotone Lagrangian of minimal Maslov class $2$, and has displacement energy $e(L_\epsilon) = \epsilon$.

For any open neighbourhood $U$ of the zero-section $Q \subset U \subset T^*Q$, $c(Q, U) \leq c( L_\epsilon, D^2(2\epsilon) \times U)$ by Lemma \ref{lem:Product ineq}, so in particular $c(L_\epsilon, D^2(2\epsilon) \times U)$ is bounded away from $0$.

Furthermore, $e(L_\epsilon, D^2(2\epsilon) \times U) \le 4 \epsilon$. 
\qed
\end{lem}

Now observe that in $\C^3$, there is an isotropic embedding of $S^2$ (e.g., take the real axis $\R^3$ and intersect with the unit sphere). The isotropic neighbourhood theorem gives a symplectic embedding of $D^2(a) \times U$ for some $a > 0$
and $U$ a neighbourhood of the 0-section in $T^*S^2$. This embedding can be made exact, and thus $L_\epsilon = S_\epsilon \times S^2$ 
embeds as a monotone Lagrangian in $\C^3$ for $\epsilon < a$. 

This then gives:
\begin{prop}
There exists a sequence of closed, monotone Lagrangians $L_\epsilon$, with minimal Maslov number $2$, and neighbourhoods $U_\epsilon$ in $\C^3$ with the following properties:
\begin{enumerate}
\item the relative capacities $c(L_\epsilon, U_\epsilon)$ are bounded away from $0$;
\item the displacement energy $d(L_\epsilon)$ and the energy to displace $L_\epsilon$ from $U_\epsilon$, $d(L_\epsilon, U_\epsilon)$ both have order $\epsilon$ as $\epsilon$ goes to $0$.
\end{enumerate}
\qed
\end{prop}
In other words, there is no possible energy-capacity inequality (even with a constant) 
in general for closed Lagrangians in $\C^3$.

Notice furthermore that for $\epsilon$ small, we can symplectically embed a neighbourhood of
$L_\epsilon \subset \C^3$ in a Darboux chart in, for instance, $\mathbb{CP}^3$. The
bounds on displacement energy and relative capacity remain valid, so the
compactness/non-compactness of the symplectic manifold doesn't matter for the
result.

\section{Filtered Floer complexes and persistence modules}

\subsection{The Lagrangian pearl complex and Lagrangian Floer homology}

In this section will recall the key features of both the Lagrangian pearl complex and of
Lagrangian Floer homology, following the sign conventions and notation from
Biran and Cornea \citelist{\cite{Biran_Cornea_pp_2007}
\cite{Biran_Cornea_new_persp}}. 
The Lagrangian pearl complex was introduced by
Oh \cite{Oh_1996} and subsequently developed extensively by Biran and Cornea 
\citelist{\cite{Biran_Cornea_pp_2007}, \cite{Biran_Cornea_new_persp}}.

The exposition in this section follows 
\cite{Biran_Cornea_new_persp}*{Sections 2.1, 2.2, 2.3} and, 
for the comparison with Floer homology, the discussion in
\cite{Biran_Cornea_pp_2007}*{Section 5.6}. Finally, the discussion of actions
is addressed in \cite{Biran_Cornea_pp_2007}*{Section 2.3}. 
We include it here for convenience of the reader, and so that our discussion remains self-contained.

Let $(W, \omega)$ be a connected, geometrically bounded symplectic manifold
(see, for instance, \cite{Audin_Lalonde_Polterovich_1994}*{Definition 2.2.6})
Let $L \hookrightarrow W$ be an closed, embedded, connected Lagrangian
submanifold, and suppose that $L$ is (spherically) monotone, i.e. that there exists a constant $\tau > 0$ so that 
for all disks $D \in \pi_2(W, L)$, $\omega(D) = \tau \mu(D)$, where $\mu(D)$ is the Maslov index
of the disk $D$. Let $N_L$ be the minimal Maslov index of a disk with positive
symplectic area. We assume furthermore that $N_L \ge 2$. 

For notational purposes, let $\bar \mu(A) \coloneqq \mu(A)/N_L$. 

\subsubsection{Lagrangian pearl complex}

Let $f \colon L \to \R$ be a Morse function and choose a metric $g$ on $L$ such
that $(f,g)$ are a Morse-Smale pair. Let $\nabla f$ denote the gradient of $f$
with respect to $g$. 
Let $\Phi_t$ denote the flow of $-\nabla f$. The unstable manifold of a
critical point $x$ with respect to $-\nabla f$ will be called the descending
manifold of $x$. Similarly, the stable manifold is called the ascending manifold. 
For the critical point $x$, we write $\Morse(x)$ for its
Morse index (i.e.~the dimension of the descending manifold of $x$).

The \defin{Quantum Homology Complex} (or \defin{pearl complex}) is generated by the critical
points of $f$ and is taken with coefficients in the Laurent polynomial 
ring $\Z_2[t^{-1}, t]$. The formal variable $t$ is graded by $|t| = -N_L$ and
critical points $p \in \Crit{f}$ are graded by $|p| = \Morse(p)$.
From this, the grading on monomials in the pearl complex is given by $|pt^r| = \Morse(p) - rN_L$.

Fix an almost complex structure $J$ on $W$ so that $(W, \omega, J)$ is
geometrically bounded. Then, a holomorphic disk $u \colon (D, \partial D) \to (W,L)$ is a map
from the unit disk in $\C$ so that $du + Jdu\circ i = 0$ and $u(\partial D) \subset L$.

For any two critical points $x, y \in \Crit{f}$, and relative homology
class $A \in H_2(W, L)$ in the image of $\pi_2(W,L)$, define a \defin{pearly trajectory} with $l \ge 1$
pearls from $x$ to $y$ representing the class $A$ to 
be the collection of $l$ holomorphic disks $u_1, u_2, \dots, u_l$ with
the following properties:
\begin{enumerate}
    \item $u_1(-1)$ is in the descending manifold of $x$;
    \item for each $1 \le i \le l-1$, there exists a positive $t_i$ so that 
        $\Phi_{t_i}(u_i(1)) = u_{i+1}(-1)$;
    \item $u_l(1)$ is in the ascending manifold of $y$.
\end{enumerate}
Each of the non-constant disks has an $\R$ action by domain automorphisms,
and thus such a configuration with $l$ pearls has 
an $\R^l$ action by domain automorphisms. We consider these configurations
modulo this action.

This problem is Fredholm, and modulo automorphisms, the Fredholm index
(virtual dimension) is given by 
\[
    \Morse(x) - \Morse(y) + \mu(A) -1.
\]
(See, for instance, \cite{Biran_Cornea_pp_2007}*{Propositions 3.1.2, 3.1.3} where this
is stated, but it follows from the Fredholm theory in
\cite{McDuff_Salamon_2004}*{Appendix C}.)

For $A=0$, we consider pearly trajectories from $x$ to $y$ to be the
negative gradient trajectories modulo its $\R$-action. This then has the same
virtual dimension formula.

The 0 and 1 dimensional moduli spaces are transverse for generic choice of data
\cite{Biran_Cornea_pp_2007}*{Section 3}.
Let $n(x, y, A)$ be the number of pearly trajectories from $x$ to $y$
representing the class $A$ (taken modulo $2$) if 
\[ \Morse(x) - \Morse(y) + \mu(A) -1 = 0 \]
and define the number to be $0$ otherwise. Then, define:
\[
    \partial x = \sum_{A \in \pi_2(W,L)} \sum_{y\in \Crit{f}}  n(x,y,A) y
    t^{\bar \mu(A)},
\]
and extend $\partial$ linearly over the coefficient ring $\Z_2[t^{-1}, t]$.
This satisfies $\partial^2 = 0$ and decreases the degree by $1$.

We denote the resulting chain complex by $CQ(L)$. Despite this notation
(chosen for brevity), it depends on the almost complex structure, 
Morse function and gradient-like vector field used to define it.
The homology of this complex is denoted by $QL(L)$, and different choices of
data give isomorphic homologies (\cite{Biran_Cornea_pp_2007}*{Proposition
5.1.2}).

The pearl complex may be informally thought of as the Floer complex associated to a constant Hamiltonian. 
Inspired by this point of view, for each choice of constant $c\in \R$,
we define 
the \defin{action} $\A_c$ of the generators of the complex as follows:
\[
\A_c(xt^r) =c -\tau r N_L.
\]
Notice that if 
$\partial x = \dots + yt^r + \dots$, then a pearly trajectory connects $x$ and
$y$, representing some class $A$ with $\mu(A) = rN_L$. By monotonicity,
$\omega(A) = \tau r N_L = -\A_c(yt^r) + \A_c(x) \ge 0$, and so the differential
does not increase action. (Equality of action occurs only when $r=0$, i.e.~when
the pearly trajectory is a pure Morse trajectory without any holomorphic disks.) 
This action filtration now allows us to consider the homologies of subcomplexes
filtered by action.

\subsubsection{Floer homology}

As above, we take $L \hookrightarrow W$ to be a monotone Lagrangian with minimal Maslov
number $2$. 
Fix a time-dependent Hamiltonian function $H \colon S^1 \times W \to \R$, 
so that $dH$ is compactly supported. 
Let $J_t, t \in S^1$ be an $S^1$-family of almost complex structures on $W$, independent
of $t$ outside a compact set, and making $(W, J_t)$ geometrically bounded for
each $t$. Define the time dependent Hamiltonian vector field $X_{H_t}$ by the equation:
\[
    \omega(X_{H_t}, \cdot) = - dH_t(\cdot).
\]

We consider the chords of $X_{H_t}$: these are trajectories $\gamma \colon [0,
1] \to W$ with $\gamma(0) \in L, \gamma(1) \in L$ and so that $\dot \gamma(t)
= X_{H_t}(\gamma(t))$. For generic $H$, all such chords are non-degenerate, i.e. if $\psi_t$ is the flow of
$X_{H_t}$, $T\psi_1(\gamma(0)) T_{\gamma(0)}L$ intersects $T_{\gamma(1)}L$
transversely. In particular, the chords are isolated and there are a finite number of them.

For each chord $\gamma$, fix a capping disk
$D(\gamma)$: this is a disk in $W$ whose boundary consists of two arcs: one along $L$ and the other following $\gamma$, oriented positively. Notice that $\pi_2(W, L)$ acts on the space of cappings by taking a connected sum, and 
the difference between two cappings is an element of $\pi_2(W, L)$.
In the case of non-degenate chords, the 
pair $x = (\gamma, D(\gamma))$ has a well-defined (integer-valued) Maslov index 
{which we denote $\mu(\gamma,D(\gamma))$}
\citelist{\cite{Viterbo_1987}
\cite{Robbin_Salamon_1993}}.

The Floer complex of $H$, assuming that all of the chords are non-degenerate, is generated by these capped orbits over the Laurent
polynomial ring $\Z_2[t^{-1}, t]$. The formal
$t$ variable then captures the action of $\pi_2(W, L)$ on the capping disks.

We define $|xt^r| = \mu( \gamma, D(\gamma) ) - rN_L$, and in particular, $|x| = \mu(\gamma,D(\gamma))$.
We may think of the monomial $xt^r$ as representing the orbit $\gamma$ capped by 
the connect sum of the reference capping disk $D(\gamma)$ with a disk 
with Maslov index $rN_L$. (To make this precise, we consider equivalence
classes of capping disks, where two disks are equivalent if they have the same
Maslov index.)

For a pair of capped orbits $(\gamma_-, D(\gamma_-))$ and $(\gamma_+,
D(\gamma_+))$, we consider solutions $u \colon \R \times [0, 1] \to W$ 
to the Floer equation:
\begin{equation}\label{eq:CR equation}
    u_s + J(t, u)(u_t - X_{H_t}(u)) = 0,
\end{equation}
satisfying the boundary conditions $u(s, 0) \in L, u(s, 1) \in L$, and the
asymptotic boundary conditions $u(+\infty, t) = \gamma_+(t)$, $u(-\infty, t) =
\gamma_-(t)$.

Note that this strip gives a rectangle whose boundary consists of two arcs on
$L$ together with arcs along $\gamma_-$ and $\gamma_+$, where $\gamma_+$ is
oriented positively and $\gamma_-$ is oriented negatively with respect to the orientation of the strip. We then may glue
this strip to the cap $D(\gamma_-)$ and to the reverse cap
 $-D(\gamma_+)$ to obtain a representative of the
homology class $A \in \pi_2(W, L)$, where $A$ is
\[
    A = D(\gamma_-) \# [u] \# ( - D(\gamma_+) ).
\]
Equivalently, we may think of this as
inducing a map from the cap $D(\gamma_-)$ to a new cap on $\gamma_+$ which is obtained as the connected sum of $A$ with the old cap $D(\gamma_+$) on $\gamma_+$.
In this way, we say that the strip $u$ represents the homology class $A$.

Solutions of Equation \ref{eq:CR equation} have a one-parameter family of domain automorphisms (by domain
translations in the $s \in \R$ coordinate). 
The Fredholm index of this problem is given by 
\begin{equation} \label{eqn:Fredholm index}
\mu(\gamma_-, D_-) - \mu(\gamma_+, D_+) + \mu(A) = |x| - |y t^{\bar
	\mu(A)}|. 
\end{equation}
We are interested in solutions that are rigid 
modulo this translation, i.e. Fredholm index one.

In this monotone case, 
for generic choice of data, i.e.~of (potentially time-dependent) almost complex structure $J$, 
and of (potentially time-dependent) Hamiltonian $H$, 
the low dimensional moduli spaces of solutions to the Floer equation \eqref{eq:CR equation}
are cut out transversely. (The 1-dimensional moduli spaces are used to define
the differential and the 2-dimensional ones are used to show that $\partial_H^2
= 0$. We will also use this fact with $s$-dependent Hamiltonians in \eqref{eq:CR
equation} to construct chain maps.)
These moduli spaces are compact manifolds with boundary, of dimension given by
the Fredholm index of Equation \eqref{eqn:Fredholm index}.
(See, for instance, \cite{Oh_1993}.)

From this, we proceed to define a differential. First, we let $n(x, y, A)$ be
the number of rigid solutions (modulo the $\R$ family of automorphisms)
representing the class $A$ if the above index is $1$, 
and setting $n(x,y, A)$ to be $0$ otherwise. We now define the differential $\partial_H:... \to ...$ by 
\[
    \partial_H x = \sum_A \sum_{y} n(x,y,A) yt^{\bar \mu(A)}. 
\]
For generic data, standard arguments \cite{Oh_1993} give that
 $\partial_H^2 = 0$ and that the differential has
degree $-1$.

We denote the resulting chain complex by $(CF(L;H),\partial_H)$, suppressing the dependence
on the almost complex structure $J$ and on the choices of capping disks for
the generators.

The \defin{action} of a capped chord $x=(\gamma, D(\gamma))$ is defined as follows:
\[
    \A_H(x) = \A_H(\gamma, D(\gamma)) = \int_{0}^1 H(t, \gamma(t)) \, dt -
    \int_{D(\gamma)} \omega.
\]
We then extend this notion to monomials in our chain complex, where $x =
(\gamma, D(\gamma))$: 
\[
    \A_H(xt^r) = \A_H(\gamma, D(\gamma)) - r \tau  N_L.
\]

Finally, we make the convention that, for an arbitrary element of the chain complex, \[
\A_H \left ( \sum_{j=1}^N x_j t^{r_j} \right ) = \max
\left \{ \A_H( x_j t^{r_j}) \, | \, \quad j = 1, \dots, N \right \} 
\]
where the monomials in the sum are distinct (i.e.~$x_j
t^{r_j} = x_l t^{r_l}$ if and only if $j = l$).

We will now recall the standard action estimate. While it is not necessary at
this point in the exposition, we will consider the more general case of a
Hamiltonian that depends not only on {$t \in S^1$} but also on a variable {$s\in \R$}.
We write $dH_{s,t}$ to denote the exterior
derivative of $H(s,t, \cdot) \colon W \to \R$ so that $dH = H_s ds + H_t dt +
dH_{s,t}$, where $H_s$ and $H_t$ denote the partial derivatives of $H$ with respect to $s$ and $t$, respectively.
We abbreviate here $S = \R
\times [0,1]$, and take $(s,t) \in S$ with $s \in \R$ and $t \in [0,1]$.  We also suppose that the dependence of $H$ on $s$ is only in a compact subset of $S$.
 Let $H_\pm \colon [0,1] \times W \to \R$ denote the corresponding limits as
$s\to \pm \infty$.
\begin{equation} 
\begin{aligned}\label{eqn:action estimate}
    0 <& \int_{S} |u_s|^2 \, ds \wedge dt 
    =  \int_S \omega( u_s, Ju_s) ds \wedge dt \\
    =& \int_S \omega(u_s, u_t - X_{H_{s,t}}(u)) \, ds\wedge dt\\ 
    =& \int_S u^*\omega - \int \omega(u_s, X_{H_{s,t}}(u)) \, ds \wedge dt \\
    =& \int_S u^*\omega - \int_S dH_{s,t}[u_s] ds \wedge dt \\
    =& \int_S u^*\omega - \int_S u^*(dH - \partial_s H ds) \wedge dt \\
    =& \int_S u^*\omega - \int_S u^*dH \wedge dt + \int_S  (\partial_s H)(s,t,u(s,t)) ds \wedge dt \\ 
    =& \int_S u^*\omega - \int_{[0,1]} H_+(t, \gamma_+(t))\,dt  + \int_{[0,1]}
    H_-(t, \gamma_-(t)) \, dt \\
     &+ \int_S  (\partial_s H)(s,t,u(s,t)) ds \wedge dt   \\
    =& \int_S u^*\omega - \A_{H_+}(\gamma_+, D_+) - \int_{D_+} \omega +
    \A_{H_-}(\gamma_-, D_-) + \int_{D_-} \omega \\
    &+ \int_S  (\partial_s H)(s,t,u(s,t)) ds \wedge dt   \\ 
    =& \;\omega(A) - \A_{H_+}(\gamma_+, D_+) + \A_{H_-}(\gamma_-, D_-) +\\
     &+ \int_S \partial_s H(s,t,u(s,t)) ds \wedge dt.
\end{aligned}
\end{equation}

In particular then, if $\partial_H x = \dots + yt^r + \dots$, then there exists a strip representing the class $A$ going from $x = (\gamma_-,
D_-)$ at $-\infty$ to $y = (\gamma_+, D_+)$ at $+\infty$ with $\mu(A) = rN_L$. 
Thus, $ 0 < \omega(A) - \A_H(\gamma_+, D_+) + \A_H(\gamma_-, D_-) 
= \tau r N_L - \A_H(y) + \A_H(x) = \A_H(x) - \A_H(yt^r).$ The differential therefore decreases action.

In the context of the Quantum Homology Complex (aka the Pearl complex), 
 it will be useful to consider $\A_0$, the
action associated to the identically zero Hamiltonian, or more generally,
$\A_c$, associated to a constant Hamiltonian.

Given two (non-degenerate) Hamiltonians $H_1, H_2$, we may define a chain map
between their Floer chain complexes by considering a continuation map, where we
count solutions to the Floer equation with an $s$-dependent Hamiltonian that
interpolates between $H_1$ and $H_2$. Note that, for the differential on our chain complex, we have that $H_s = 0$.
In general, we say that a homotopy is a \defin{monotone homotopy} if $H_s \le 0$ pointwise.
By the action estimate \eqref{eqn:action estimate}, 
this continuation map will be action non-increasing if we have a monotone
homotopy (i.e.~$H_s \le 0$ pointwise).
More generally, the action estimate can also be applied to a linear
interpolation between the two Hamiltonians. 
In this case, we set $H = (1-\rho(s)) H_- + \rho(s) H_+$, 
where $\rho \colon \R \to [0,1]$ has and $\rho'$ is non-negative and compactly supported. 
Then, $H_s = \rho'(s) \left ( H_+ - H_- \right )$. 
This more general PSS-type map will thus increase action by at most
\begin{equation} \label{eqn:Eplus}
    \E_+(H_+ - H_-) = \int_{0}^1 \sup \{ H_+(t, x) - H_-(t, x) \, | \, x \in W \} \, dt.
\end{equation}

\subsubsection{Chain maps and consequences}

We now introduce two important chain maps, as discussed in
\cite{Biran_Cornea_pp_2007}*{Proposition 5.6.2}:
\begin{equation} \label{eqn:chain maps Psi and Phi}
\begin{aligned}
    \Psi \colon CQ(L) \to CF(L;H) \\
    \Phi \colon CF(L; H) \to CQ(L).
\end{aligned}
\end{equation}

These are obtained by counting rigid configurations of PSS-type mixed
holomorphic and Floer solutions.  
These two maps are chain maps, and their compositions are chain homotopic to
the identity.

Let $CQ^{\le a}$ denote the subcomplex whose generators have $\A_0$ action at
most $a$. 
Let $CF^{\le a}$ denote the subcomplex of action at most $a$.
By the action estimate \eqref{eqn:action estimate}, any PSS-type map will be
action non-increasing \textit{for a monotone homotopy} (i.e.~when $H_s \le 0$
pointwise).
Hence, whenever the
Hamiltonian satisfies $H \le c_1$ pointwise, we have that 
$\Psi \colon CQ(L) \to CF(L; H)$ is action non-increasing when the action on
$CQ$ is taken to be $\A_{c_1}$. 
Indeed, in that case, we may think of $CQ(L)$
as the Morse-Bott Floer complex for the constant Hamiltonian $c_1$. 

If, instead, $H \ge c_2$ pointwise, we have that $\Phi \colon CF(L; H) \to CQ(L)$
is action non-increasing when $CQ$ is equipped with the action $\A_{c_2}$.

Combining these observations, for any Hamiltonian $H$ with $c_1 \le H \le c_2$, we
obtain that 
\begin{equation} \label{eqn:action estimate chain maps}
    CQ^{\le a}(L) \stackrel{\Psi}{\to} CF^{\le a+c_2}(L;H) \stackrel{\Phi}{\to} CQ^{\le a+c_2 - c_1}(L)
\end{equation}
is chain homotopic to the inclusion $CQ^{\le a} \to CQ^{\le a+c_2 - c_1}$. From this, we obtain a version of Chekanov's theorem:
\begin{prop}[\cite{Biran_Cornea_pp_2007}*{Remark 6.6.3}] \label{prop:chekanov}
    If $L$ is displaceable with displacement energy $d(L)$ then the inclusion $CQ^{\le a}
    \to CQ^{\le a + d(L)}$ is chain homotopic to $0$. 

    In particular, if the Morse function on $L$ has a unique local maximum, $M \in
    \Crit(f)$, then $M = \partial \sum_{i=1}^l x_i t^{r_i}$ where 
    $\A_0(x_i t^{r_i}) \le d(L)$.
\end{prop}

Notice that for degree reasons, the exponents $r_i < 0$ and thus, in particular,
there is a holomorphic disk of area at most $d(L)$ passing through a generic
point.

\subsection{Persistent homology for admissible Hamiltonians}

In the following, we continue to work with our standing assumptions that $L$ is a closed, monotone
Lagrangian embedded in $(W, \omega)$ with minimal Maslov number $N_L \ge 2$, 
and monotonicity constant $\tau > 0$. We
	considering $CF(L; H)$ to be a complex of $\Z/2$ vector spaces with an infinite
	generating set $\{ t^k y \, | \, y \text{ is a chord} \}$.

\subsubsection{The r2p persistence module in the non-degenerate case}

First, suppose that $H$ is a non-degenerate Hamiltonian, and, in particular
then, not admissible. As discussed after Equation \ref{eqn:action estimate}, the differential on the Floer
complex $CF(L; H)$ is action decreasing. We thus define, for $a \in \R$,
$CF^{<a}(H)$ to be the subcomplex of action less than $a$, and similarly
$CF^{\le a}(H)$ to be the subcomplex of action no greater than $a$. 
Associated to an interval $(a,b)$ we then consider the quotient complex
\[
    CF^{(a,b)}(H) = CF^{<b}(H) / CF^{\le a}(H).
\]
Notice that we have suppressed the Lagrangian $L$ from this notation.
While we've defined this as a quotient complex, it is also obtained by
considering the generators whose action is in the ``window'' $(a, b)$. 
We then denote by 
\[
    HF^{(a,b)}(H)
\]
the homology of this complex. 
Let $\Sigma(H)$ denote the action spectrum of the Lagrangian chords of $H$.

Let $a_0 = \tau N_L$ denote the symplectic area of any disk in $\pi_2(W, L)$ whose Maslov number is minimal. 
Multiplication by the monomial $t$ induces an isomorphism 
\begin{equation} \label{eqn:t multiplication}
    t \colon HF^{(a,b)}(H) \to HF^{(a-a_0, b-a_0)}(H)
\end{equation}
that shifts the grading by $-N_L$.
Following Polterovich and Shelukhin \cite{Polterovich_Shelukhin_2016} and their
notation, this gives an r2p persistence module. 

Let $\Intervals(H) = \{ (a, b) \subseteq \R \, | \, a, b \notin \Sigma(H) \}$.
Define a partial order on $\Intervals(H)$ by $(a,b) \preceq (a', b')$ if $a \le
a'$ and $b \le b'$. Then, for two intervals $I_1, I_2 \in \Intervals(H)$, if
$I_1 \preceq I_2$, we have the map induced by inclusion:
\[
    \imath_{I_2, I_1} \colon HF^{I_1} \to HF^{I_2}.
\]
This has the naturality property that if $I_1 \preceq I_2 \preceq I_3$ then 
\[
    \imath_{I_3, I_1} = \imath_{I_3, I_2} \circ \imath_{I_2, I_1}.
\]
Associated to $a< b< c$ with $a, b,c \notin \Sigma(H)$, we have an exact
triangle:
\begin{equation} \label{eqn:exact triangle}
    HF^{(a,b)}(H) \xlongrightarrow{\imath} 
    H^{(a,c)}(H) \xlongrightarrow{\imath} 
    H^{(b, c)}(H) \xlongrightarrow{\delta} 
    HF^{(a,b)}(H)[1] 
\end{equation}
where the bracket $[1]$ indicates that $\delta$ lowers the degree by $1$.
The map $\delta$ can be constructed from the observation that $CF^{I}(H)$ is
generated by the chords whose action is in the window $I$. From this, obtain
that $CF^{(a,c)}(H) \cong CF^{(a,b)}(H) \oplus CF^{(b,c)}(H)$.
The differential on $CF^{(a,c)}(H)$ is upper triangular with respect to this
splitting, and $\delta$ is the map induced by the off-diagonal term (which is
then a chain map). 

An important special case gives that if $a < a' < b' < b$ with $[a, a'] \cap \Sigma(H) = \emptyset = [b', b] \cap \Sigma(H)$,
then (using the triples $a'< b' < b$ and $a<a'<b$)
\begin{equation} \label{eqn:shrink interval}
    HF^{(a, b)}(H) \cong HF^{(a', b')}(H).
\end{equation}

Finally, given two Hamiltonians $H_1, H_2$, we consider the continuation map
from $CF(L; H_1)$ to $CF(L; H_2)$ induced by the linear interpolation between
them.
As a consequence of the action estimates in Equations \ref{eqn:action
estimate} and \ref{eqn:Eplus}, we have that the continuation map gives a map
\begin{equation} \label{eqn:filtered continuation}
    HF^{(a,b)}(H_1) \to HF^{(a+C, b+C)}(H_2)
\end{equation}
where $C = \E_+( H_2 - H_1)$.

\subsubsection{The r2p persistence module in the degenerate case}
We will now construct the filtered homology groups $HF^{(a,b)}(H)$ for the
	admissible Hamiltonian $H$, despite the fact that it is degenerate and hence
	doesn't have a corresponding chain complex.
	In this we are following the ideas from 
	\cite{Polterovich_Shelukhin_2016}.

	Let $H \colon W \to \R$ be admissible in the sense of Definition \ref{def:modified_admissible}.
The action spectrum of the Lagrangian chords of $H$ is then given by 
\[
    \Sigma(H) = \{ k a_0 \, | \, k \in \Z \} \cup \{ m(H) + k a_0 \, | \, k \in \Z \}. 
\]
We impose the additional condition that 
\begin{equation} \label{eqn:nonresonance m(H) assumption}
    \frac{m(H)}{a_0} \notin \Z.
\end{equation}

We will now construct the filtered homology groups $HF^{(a,b)}(H)$ for the
admissible Hamiltonian $H$, despite the fact that it is degenerate.
In this we are following the ideas from
	\cite{Polterovich_Shelukhin_2016}.
Consider an interval $(a,b)$ such that $a, b \notin \Sigma$. 
Then, fix a $\delta> 0$ so that the intervals $[a-2\delta, a+2\delta]$ and
$[b-2\delta, b+2\delta]$ are disjoint from $\Sigma$.
Let $H_1, H_2$ be two non-degenerate Hamiltonians that are $C^2$ close to $H$:
$\Vert H_j - H \Vert_{C^2} < \delta/2$ and hence to each other.
Their action spectra are therefore also close.
By considering a continuation map, 
we obtain chain homotopy equivalences $CF(L; H_1) \to CF(L; H_2)$ and $CF(L; H_2) \to CF(L;
H_1)$ that are chain inverses of each other.
By Equation \ref{eqn:action estimate chain maps}, each of these maps increases action by at
most $\delta$, so we obtain maps
\[
    HF^{(a, b)}(H_1) \to HF^{(a+\delta, b+\delta)}(H_2) \to HF^{(a+2\delta,
    b+2\delta)}(H_1) \equiv HF^{(a, b)}(H_1).
\]
The composition of these maps is the identity map, since at the chain level the composition is
chain homotopic to the map induced by inclusion. 
This then gives that the following is well-defined (i.e.~independent of the
choice of $H_1$):
\begin{defn} \label{def:degenerate filtered Floer}
    Let $H$ be an admissible Hamiltonian and let
    $(a,b) \subseteq \R$ such that $a, b \notin \Sigma$.

   Let $H_1$ be any non-degenerate perturbation of $H$ that is sufficiently
   $C^2$ close to $H$. We define
   \[
       HF^{(a,b)}(H) = HF^{(a,b)}(H_1).
   \]
\end{defn}

By a similar argument, it follows that if $H_2 \le H_1$ pointwise, there are
natural maps induced by (a perturbation of) the monotone homotopy 
\[
    \Psi_{H_1,H_2,a,b}:HF^{(a,b)}(H_1) \to HF^{(a,b)}(H_2)
\]
as long as $a, b$ are not in the action spectrum of either Hamiltonian.
Finally, by viewing $CQ(L)$ as the Morse-Bott version of $CF$ for the zero
Hamiltonian, note that we obtain the following
\[
    HF^{(a, b)}(0) \cong H_*(L; \Z/2)
\]
for any $a \in (-a_0, 0)$ and $b \in (0, a_0)$.

\subsection{Local Floer homology}

In this section, we follow some of the ideas of local Floer homology
\citelist{\cite{Ginzburg_2010}, \cite{Ginzburg_Gurel_2009}}. Notice that these
references are concerned with Hamiltonian Floer homology, whereas we are looking
at Lagrangian Floer homology (with a Hamiltonian term). The changes are minimal,
but we will emphasize those points particularly.

The main result of this section is the following:
\begin{thm} \label{thm:filtered homology of the maximum}
    Let $H$ be an admissible Hamiltonian such that $m(H)/a_0 \notin \Z$, where $a_0$ is .
    Let $\delta > 0$ be sufficiently small so that $[m(H) - \delta, m(H)+\delta] \cap \Sigma(H) = \{ m(H) \}$. 
    Then,
 \[
     HF^{(m(H)-\delta, m(H)+\delta) }(H) \cong \Z/2[n].
 \]
 (We use $\Z/2[n]$ to denote that this has grading $n$.)

 Furthermore, if $H_1, H_2$ are two such admissible Hamiltonians, the
 continuation map (induced by linear interpolation between them) induces an
 isomorphism 
 \[
HF^{(m(H_1)-\delta, m(H_1)+\delta) }(H_1)  \to 
HF^{(m(H_2)-\delta, m(H_2)+\delta) }(H_2) 
 \]
\end{thm}

In order to prove this, we will need some preliminary lemmas to describe the
structure of the local Floer homology. Specifically, we will show that a
complete understanding of $HF^{(m(H)-\delta, m(H)+\delta)}(H)$, for the
differential and also for all relevant chain maps, is given by
studying curves contained in a neighbourhood of \[
    \overline{D} \coloneqq \{x \in L \, | \, H(x) = m(H)
\}.\]

\begin{lem}
    \label{lem:local Floer supported in U}

    Let $H$ be an admissible Hamiltonian such that $m(H)/a_0 \notin \Z$, and let $\delta > 0$ be a real number which is sufficiently small so that 
$[m(H) - \delta, m(H)+\delta] \cap \Sigma(H) = \{ m(H) \}$. Let $U$ be an open neighbourhood of $\overline{D} \coloneqq \{ x \in L \, | \, H(x) = m(H) \}$.

Then, for any $\hat H$ sufficiently $C^2$--close to $H$, 
$HF^{(m(H)-\delta, m(H)+\delta)}(H)$ is obtained as the homology of the
chain complex
\[
    CF^{(m(H)-\delta, m(H)+\delta)}(\hat H),
\]
all of the generators of this complex are contained in the open set $U$, and the
differential counts Floer trajectories contained entirely in $U$.

Furthermore, if $\hat H_1, \hat H_2$ are two sufficiently small perturbations of
$H$, the chain maps between $CF^{(m(H)-\delta, m(H)+\delta)}(\hat H_1)$ and 
$CF^{(m(H)-\delta, m(H)+\delta)}(\hat H_2)$ count Floer continuation trajectories 
contained in $U$.

\end{lem}

\begin{proof}

    Let $U$ be an open neighborhood of $\overline{D}$. By shrinking $U$ as necessary, we may assume that $U$ is contained in a Weinstein neighbourhood
    of $L$. We may choose a Hamiltonian $\hat H$ sufficiently $C^2$-close to $H$ so that the chords of $\hat{H}$ are close to the
    chords of $H$. These chords of $\hat H$ therefore admit capping disks that are
    contained in $U$ and have a well-defined action with
    respect to such a capping disk. For each chord of $\hat H$, we fix such a capping. Now note that, since $m(H)/a_0 \notin \Z$ and for sufficiently small $\delta$, the capped chords of $H$ in the action window $I(m(H),\delta)$  are the constants $\{ x \in
    L \, | \, H(x) = m(H) \}$ with constant capping disks, and hence for $\hat H$ sufficiently close to $H$,
    the capped chords of $\hat H$ in the action window $I(m(H),\delta)$ are also contained in $U$.

    Note that any Floer trajectory between two chords of $\hat H$ in the action window $I(m(H),\delta)$ must have
    small energy. In the limit $\hat H \to H$, the energy is too low (i.e. below $a_0$) for bubbling
    of any kind, and hence a Floer solution between chords of $\hat H$ in the action window $I(m(H),\delta)$ will converge to
    a Floer solution with zero energy, and hence to a constant. For $\hat H$
    sufficiently $C^2$-close to $H$ then, the Floer trajectory is converging in $C^{\infty}_{loc}$ to constant trajectories in $\overline{D}$, and so it will also be contained
    entirely in the neighbourhood $U$.
   
    By a similar argument, for $\hat H_1$ and $\hat H_2$ sufficiently $C^2$-close to $H$, a Floer continuation trajectory between a chord of
    $\hat H_1$ and a chord of $\hat H_2$, both in the action window $I(m(H),\delta)$ will also have
    small energy, and in particular energy less than $a_0$, so no bubbling can occur. Similarly, no bubbling can occur in the limit $\hat H_j \to
    H$, and the Floer continuation trajectory converges in $C^{\infty}_{loc}$ to constant trajectories in $\overline{D}$. Hence,
    for $\hat H_j$ sufficiently $C^2$-close to $H$, the continuation trajectory will
    also lie entirely in $U$.
\end{proof}

\begin{lem} \label{lem:only care about U}

    Suppose that $H_0, H_1$ are admissible Hamiltonians with $m(H_j)/a_0
    \notin \Z$ for $j=0,1$. 
 Let $\delta > 0$ be sufficiently small that 
 \[ 
     [m(H_j) - \delta, m(H_j)+\delta ] \cap \Sigma(H_j) = \{ m(H_j) \}
 \] 
 for $j=0,1$. For $j = 0, 1$, write
    \[
        \overline{D}_j \coloneqq \{ x \in L \, | \, H_j = m(H_j) \}.
    \]

    If $\overline{D}_0 = \overline{D}_1$ and $dH_0 = dH_1$ in an open
    neighbourhood $U$ of $\overline D\coloneqq \overline{D}_0 = \overline{D}_1$, then
    \[
        HF^{(m(H_0)-\delta, m(H_0) + \delta)}(H_0)  \cong 
        HF^{(m(H_1)-\delta, m(H_1) + \delta)}(H_1)  
    \]
\end{lem}

\begin{proof}
    By Lemma \ref{lem:local Floer supported in U}, the filtered Floer homology
    groups 
    $HF^{I(m(H),\delta)}(H_j)$, $j=0,1$
    is computed entirely 
    by considering generators and curves contained in $U$. For a generic $C^2$-small
    $f:W \to \R$ supported in a neighborhood of $L$, both $H_0 + f$ and $H_1 + f$ are nondegenerate, and since the
    differentials are equal in $U$ by hypothesis, $dH_0 = dH_1$ in $U$, and therefore the complexes $CF^{I(m(H_0),\delta)}(H_0+f)$ and $CF^{(I(m(H_1,\delta))}(H_1+f)$ can be identified, and the result follows. 
\end{proof}

Given an admissible $H_0$, we will now construct a ``small'' admissible
Hamiltonian $H_1$ whose local homology relative to a small interval around its maximum is isomorphic to that of $H_0$.

\begin{lem} \label{lem:exists small admissible representative}
    Let $H_0$ be an admissible Hamiltonian such that $m(H_0)/a_0 \notin \Z$.  Let $0< \alpha < m(H_0)$ and $\alpha < a_0$, and choose $0 < \delta < \frac{1}{2} \alpha$ small enough so that both
\[ 
     [m(H_0) - 2\delta, m(H_0)+2\delta ] \cap \Sigma(H_0) = \{ m(H_0) \}
 \] 
 and 
\[ 
    [\alpha - 2\delta, \alpha+2\delta ] \cap \left ( a_0 \Z \cup (a_0 \Z +
    \alpha) \right )  = \{ \alpha \}.
 \] 

    Then there exists an admissible Hamiltonian $H_1$ such that $m(H_1) = \alpha$ and so that 
    \[
        HF^{(m(H_0) -\delta, m(H_0) + \delta)}(H_0) \cong HF^{(m(H_1)
        -\delta, m(H_1) + \delta)}(H_1).
    \]
    Furthermore, by choosing $\alpha$ sufficiently small, $dH_1$ can be made arbitrarily small as well (in the $C^0$ topology).
\end{lem}

\begin{proof}

    We may assume that $\alpha < m(H_0)$, since otherwise the result is
    trivially true

    Since $H_0$ is admissible, any non-constant chord has return time $T>1$. Observe though that by the Arzela-Ascoli Theorem (\cite{Dunford_Schwartz_1958}, Theorem IV.6.7), the set of return times of chords of $H_0$ is a
    closed set, so there exists $\eta > 0$ so that all chords of length at most
    $1+\eta$ are constant.

    Let $a = m(H_0) - \alpha$ and $b = m(H_0) - 2\epsilon$, and note that $a < b$. Let $\rho \colon \R \to \R$ be a smooth function with the following
    properties:
    \begin{itemize}
        \item $\rho(s) = a$ for all $s \le a$
        \item $\rho(s) = s$ for all $s \ge b$
        \item $\rho'(s) \le 1+\eta$ for all $s$
        \item $\rho'(s) > 0$ for $s> a$.
    \end{itemize}

    Consider now the function \[
        H_1 = \rho\circ H_0 - a = \rho \circ H_0 - m(H_0) + \alpha.
    \]
    We claim that $H_1$ is
    admissible. To see this, we first calculate 
    \begin{equation*}
    	H_1(x) = \begin{cases} 
    		0 & H_0(x) \leq m(H_0)-\alpha \\
    		H_0(x) - m(H_0) + \alpha & H_0(x) \geq m(H_0) -2\epsilon,
    		\end{cases}
    \end{equation*}
    and we note that in the region $\{x \in W \mid H_(x) \geq m(H_0) - 2\epsilon \}$, $H_0(x) - m(H_0) + \alpha \ge \alpha -2\epsilon > 0$ since, by hypothesis, $0 < \alpha < \frac{1}{2}a_0$. Since $\rho'(s) > 0$ for $s > a$, it follows that $H_1 > 0$ for $\{x \in W \mid a < H_0(x) < b\}$. Combined with the above, we conclude that $H_1 \ge 0$ pointwise.
    
Now note that
    \[
        X_{H_1} = \rho'(H_0) X_{H_0},
    \]
    and therefore the trajectories of $H_1$ are reparametrizations of trajectories of
    $H_0$. Furthermore, the chords of $H_1$ of return time at most $1$ correspond to
    chords of $H_0$ of return time at most $1+\eta$, and hence are constant. 

    By construction, set of $W$ where $H_1$ attains its maximum is the same as the corresponding set for $H_0$, and hence
    its restriction to $L$ is diffeomorphic to a closed ball.
    Finally, notice that since $dH_1 = \rho'(H_0) \, dH_0$ 
    and $\rho \circ H_0 =0$
    on the set $\{x \in W \mid H_0(x) \le a\}$, 
    the critical points of $H_1$ are either critical points of $H_0$ or are any
    points at which $H_0 \le a$. Thus, the critical values of $H_1 \vert_L$ has are $0$ and
    $m(H_0) -a = \alpha = m(H_1)$. This establishes that $H_1$ is admissible.

    By definition, \[
        m(H_1) = m(H_0) - a = \alpha,
    \]

    and, we observe that, by construction of $H_1$,
    the differentials of $H_0$ and $H_1$ agree on the set
    $H_0^{-1}((m(H_0) - 2\epsilon, +\infty))$, which is an open neighbourhood of
    $\overline{D} = \{ x \in L \, | \, H_0(x) = m(H_0) \}$. 
    Hence, by Lemma \ref{lem:only care about U}, the first part of the result now follows.

    To see the second part, observe that $dH_1 = 0$ except when $H_0 > m(H_0) - \alpha$.  
    Let $S_\alpha = \{ x \in W \, | \, H_0(x) > m(H_0) - \alpha \}$. By the
    compact support of $H_0$, {$\overline{S_\alpha}$ is also} compact. We therefore have 
    \[
        \max |dH_1| \le (1+\eta) \max |dH_0||_{\overline{S_\alpha}}.
    \]
    Now, $dH_0 = 0$ on $\cap_{\alpha > 0} S_\alpha$, so by choosing $\alpha$
    sufficiently small, $dH_1$ can be made arbitrarily small on some open neighborhood of $\overline{D}$.
\end{proof}

\begin{lem} \label{lem:admissible from function}

    Let $f \colon L \to [0, +\infty)$ be a smooth function with two critical values, $0$
    and $m < a_0$, achieving its maximum on a set diffeomorphic to a closed
    ball.

    If $f$ is sufficiently $C^1$ small, then 
    there exists an admissible $\tilde f \colon W \to \R$ such that
    $\tilde f\vert_L = f$, $0 \le \tilde f \le m$, $\tilde f$ is supported in a neighbourhood of $L$ and its 
    low-energy Floer homology satisfies:
    \[
        HF^{(a,b)}(\tilde f) = \Z_2[n]
    \]
    for some $a, b, $ regular values of $f$ with $ 0 < a < m < b < a_0$. 
\end{lem}

\begin{proof}
    
    Fix a metric on $L$. Then for some $R>0$, there is a fibrewise $R$-ball  neighbourhood $N$ of
    the $0$-section in $T^*L$ diffeomorphic to a Weinstein neighbourhood $X \subset W$ of the
    Lagrangian $L \subset X$. Fix a diffeomorphism $\psi:X \to N$ from $X$ to $N$.

    For $f$
    sufficiently $C^1$ small, the graph of $df:L \to T^*L$ is contained in this
    Weinstein neighbourhood. Let $V$ be a fibrewise $r$-ball neighbourhood so
    that $r < R$ and the graph of $df$ is contained in $V$.

    Define a smooth function $\phi \colon [0, +\infty) \to [0,1]$ so that $\phi(u) = 1$
    for all $u \le r$ and so that $\phi(u) = 0$ for all $u \ge R$. 

    Then, the function $\widetilde{f}(p, q) = \phi(|p|) f(q)$ defines an extension
    of $f$ from $L$ to all of $W$. Observe that the time-1 flow of the
    Hamiltonian vector field of $\widetilde{f}$ sends $L$ to the graph of $-df$. 
    Thus, the chords correspond precisely to the critical points of $f$ and are therefore constant, and so the
    resulting extension is thus admissible.
    Let $g:L \to \R$ be a Morse function on $L$ which is sufficiently $C^2$-close to $f$, so that
    \begin{enumerate}
    	\item is sufficiently small that 
    	the graph of $-dg$ also stays in the neighbourhood $V$
    	\item $a$ and $b$ are regular values for $g$,
    	\item any
    	critical point $x \in L$ of $g$ with critical value $c_x \in (a,b)$ is contained in an open neighborhood of $x$ close to $\{x \in L \mid f(x) = m \}$.
    \end{enumerate}  
    Notice also that $f$ has a
    Hessian that is identically $0$ over its maximum set. We may then also
    assume that $g$ has a Hessian whose eigenvalues are small but non-zero at
    all critical points. Then, constructing
    $\widetilde{g}:W \to \R$ by
    \begin{equation*}
    	\widetilde{g}(x) = \begin{cases}
    		\phi(|p|)g(q) & x \in X, (p,q) = \psi(x) \\
    		0 & x \notin X
    	\end{cases}
    \end{equation*} gives a Hamiltonian on $W$
    whose Lagrangian chords are non-degenerate, and in bijective correspondance to critical
    points of $g$. Furthermore (by the smallness of the Hessian condition), 
    the Morse indices of the critical points of $g$
    are the Maslov indices of the corresponding Lagrangian chords for
    $\widetilde{g}$. Also by construction, $\widetilde{g}$ and $\widetilde{f}$
    are $C^2$-close.

    We now observe that in the action window $(a,b)$ where $0< a< m < b<a_0$, 
    the Lagrangian chords of $\widetilde{g}$ correspond to those critical points
    of $g$ obtained as perturbations of the degenerate family of maxima of $f$. 
    In particular then, they all have action close to $m$ and hence
    the Floer trajectories for $\widetilde{g}$ connecting two such chords have small
    energy. As we consider perturbations $\widetilde{g}$ approaching
    $\widetilde{f}$, these trajectories will converge in $C^\infty_\text{loc}$ 
    to constant trajectories. Therefore, for
    $\widetilde{g}$ sufficiently close to $\widetilde{f}$, all of the Floer trajectories between chords of $\widetilde{g}$ in the action window $(a,b)$ will be contained
    entirely in the neighbourhood $V$.

    We therefore have that the Lagrangian Floer homology $HF^{(a,b)}(L;\widetilde g)$ is
    precisely the Morse homology $H^{(a,b)}(L;g)$ of $L$ in the action window $(a,b)$ computed using the Morse function $g$. 
    By standard arguments, this is then the Morse homology of
    $g$ also taken in this action window $(a,b)$. By hypothesis, the critical
    set $\{ f(x) = m \, | \, x \in L \}$ is diffeomorphic to a closed ball, so (by the
    tubular neighbourhood theorem and using the gradient flow), 
    $\{ f(x) > a \, | \, x \in L \}$ is
    diffeomorphic to an open ball for $a>0$, since the only critical values of $f$ are $0$ and $m$. Therefore, the filtered Morse homology $H^{(a,b)}(L;g)$ in the window $(a,b)$ is isomorphic to the reduced homology $\tilde{H}_*(S^{2n};\Z/2) = \Z/2[n]$ - this follows from the definition of the Morse homology $H^{(a,b)}(L;g)$ and Proposition 2.22 in \cite{Hatcher_2002} - the result follows.
\end{proof}

\begin{lem} \label{lem:Floer homology from Morse}

Let $H_0:W \to \R$ be an admissible Hamiltonian that is sufficiently $C^1$-small so that an admissible extension $\widetilde{H_0|_L}:W \to \R$ of $H_0|_L:L \to \R$ as constructed in Lemma \ref{lem:admissible from function} exists. We also suppose that $3 m(H_0) < a_0$. Denote by $F_0 = \widetilde{H_0 \vert_{L}}$ the extension of $H_0|_L$ constructed in Lemma \ref{lem:admissible from function}. 

Then, for $0 < a < m(H_0) < b < a_0$, 
\[
    HF^{(a,b)}(F_0) \cong HF^{(a,b)}(H_0)
\]

\end{lem}

\begin{proof}

    Let $0 < a < m(H_0) < b < a_0$, and write $ C = m(H_0)$. Notice that $\max |H_0 - F_0| \le C$. 
    Hence, for any $a < b$, from the standard energy estimate (Equation \ref{eqn:action estimate chain maps}),
    the following
    continuation maps increase action by at most $C$:
    \begin{align*}
        \Phi&:HF^{(a, b)}(F_0) \to HF^{(a+C, b+C)}(H_0), \\
        \Psi&:HF^{(a, b)}(H_0) \to HF^{(a+C, b+C)}(F_0).
    \end{align*}
    We claim, however, that in the action window $(a,b)$, these continuation maps are in fact \textit{action non-increasing}. 

    We now prove this claim. By hypothesis and construction of $F_0$, the time-$1$
    chords of both $F_0$ and $H_0$ in the action window $(a, b)$ are precisely the
    constant chords of $\overline{D} = \{ x \in L \, | \, H_0(x) = m(H_0) \}$. 

    Consider now a sequence of perturbations of $F_0, H_0$, which we denote by
    $F_n, H_n$, so $F_n \to F$, $H_n \to H_0$ in $C^\infty$. For large $n$, 
    the time-$1$ chords of $F_n, H_n$ have action close to $C$ and $0$ (when
    capped by approximately constant capping disks), and occur 
    near $\overline{D} = \{ x \in L \mid H_0(x) = m(H_0)\}$ and near $\{ x \in L \, | \, H_0(x) = 0 \}$. 
    In particular then, for any $\delta > 0$, for $n$ sufficiently large, 
    all of the capped chords of $F_n, H_n$ with arbitrary capping have action in the intervals:
    \begin{equation} \label{eqn:perturbed action}
      \Sigma(F_n), \Sigma(H_n) \subseteq  \bigcup_{k \in \Z} (ka_0 - \delta, ka_0 + \delta) \cup \bigcup_{k \in
        \Z} (C+k a_0 -\delta, C+ka_0 +\delta).
    \end{equation}
    By hypothesis, $C  < \frac{a_0}{3}$.  In particular, then, for $\delta$
    small and any $\alpha$ so that $C +\delta < \alpha <
    a_0 -\delta $, 
    \begin{equation}\label{eqn:Squish action}
    	CF^{<\alpha}(F_n; L) = CF^{<C+\delta}(F_n; L),
    \end{equation}
    and similarly for $H_n$.

    Let $\Phi_n, \Psi_n$ be the continuation maps between the chain complexes:
    \begin{align*}
        \Phi_n &\colon CF(F_n; L) \to CF(H_n; L) \\
        \Psi_n &\colon CF(H_n; L) \to CF(F_n; L) 
    \end{align*}
    For $n$ sufficiently large, these chain complexes are the ones that allow us
    to define $HF^{(a,b)}(H_0), HF^{(a,b)}(F_0)$ and $\Phi_n$, $\Psi_n$ induce the maps between them.

    Let now $0 < a < C < b < a_0$, as in the hypotheses of the lemma. In order
    to prove the result, we need to
    show that for $n$ sufficiently large, $\Phi_n$ maps $CF^{<b}(F_n; L)$ to $CF^{<b}(H_n; L)$ and that it
    also maps $CF^{<a}(F_n; L)$ to $CF^{<a}(H_n; L)$. 
    We also need to show a similar result for $\Psi_n$. We will do the case of
    $\Phi_n$ in detail, since the case of $\Psi_n$ is nearly identical.

    From Equation \eqref{eqn:Squish action} with $\alpha = b$, we have (for $\delta > 0$ small, and $n$ sufficiently large) 
    $CF^{<b}(F_n; L) = CF^{<C+\delta}(F_n; L)$. 
    Also for $n$ sufficiently large, $\max |H_n - F_n| < C + \delta$.
    Hence, by the energy estimate \ref{eqn:action estimate chain maps}, we have
    \[
        \Phi_n \colon CF^{<C+\delta}(F_n; L) \to CF^{<2C+2\delta}(H_n; L)
    \]
    Now, $C < \frac{a_0}{3}$, so for $\delta$ small enough $2C + 2\delta < a_0 -
    \delta$.
    Combining this with Equation \ref{eqn:Squish action}, we have
    \[
        CF^{<2C+2\delta}(H_n; L) = CF^{<C+\delta}(H_n; L) = CF^{<b}(H_n; L).
    \]
    This then establishes
    \[
        \Phi_n \colon  CF^{<b}(F_n; L) \to CF^{<b}(H_n; L).
    \]

    We now need to show that $\Phi_n$ maps $CF^{<a}(F_n; L) \to CF^{<a}(H_n; L)$. 
    We will proceed by contradiction. Assume that for a subsequence of $n
    \to \infty$, there exists
    a generator $\gamma_n \in CF^{<a}(F_n; L)$ for which $\Phi_n(\gamma_n) =
    \gamma_n' + \dots$ where $\gamma_n' \in CF(H_n; L)$ has action at least $a$. 

    By Equation \eqref{eqn:perturbed action}, it thus follows that
    $\A_{F_n}(\gamma_n) < \delta$ and $\A_{H_n}(\gamma_n') > C-\delta$.
    Applying the action estimate (Equation \eqref{eqn:action estimate}), we have
    \[
        0 < \A_{F_n}(\gamma_n) - \A_{H_n}(\gamma'_n) + C + \delta
    \]
    In particular then, 
    \begin{equation}\label{eqn:Action reasons}
    	\begin{aligned}
    	\A_{F_n}(\gamma_n)& > -2\delta \\
    	\A_{H_n}(\gamma'_n)& < C + 2\delta.
    	\end{aligned}
    \end{equation}
    
    Let $u_n \colon \R \times [0,1] \to W$ be the continuation map solutions
    that asymptote onto $\gamma_n, \gamma_n'$ at $\mp \infty$. Recall that
    such solutions satisfy the following equation:
    \begin{equation} \label{eqn:continuation map pde}
        \partial_s u_n + J(t, u_n) \partial_t u_n + \rho(s) \nabla H_n(t, u_n)
        + (1- \rho(s) ) \nabla F_n(t, u_n) = 0
    \end{equation}
    where $\rho(s)$ has $\rho(s) = 0$ for $s \le -1$, 
    $\rho(s) = 1$ for $s \ge 1$ and $0 \le \rho'(s) \le 1$ for all $s \in \R$.

    Notice that, by Equation (\ref{eqn:action estimate}) the energy of $u_n$ has:
    \begin{equation} \label{eqn:Energy of u_n}
        \begin{aligned}
            0 \le \int_{\R \times I} \Vert \partial_s u_n \Vert_J^2 ds\wedge dt 
            &\le 
            \A_{F_n}(\gamma_n) - \A_{H_n}(\gamma'_n) + \Vert H_n - F_n \Vert
            \\
            & < \delta - (C-\delta) + (C + \delta) = 3\delta.
        \end{aligned}
    \end{equation}

    In particular then, as $n \to \infty$, this goes to $0$.

    Since the energy of the strips $u_n$ is small (and in particular below the
    area of the disk $a_0$), we cannot bubble a disk or a sphere. Thus, after
    passing to a subsequence, we have that $u_n$ will converge in
    $C^\infty_\loc$ to a function $u_0$. 

    Observe also that since for $\delta > 0$ and $n$ sufficiently large
    \begin{align*}
        &-2\delta < \A_{F_n}(\gamma_n) < \delta \\
        &C-\delta < \A_{H_n}(\gamma'_n) < C+2\delta
    \end{align*}
    for each (large) $n$, there exists $\sigma_n \in \R$ so that 
    \[
        H_0( u_n(\sigma_n, 0) ) = \frac{C}{2}.
    \]

    We now consider three cases, depending on whether (1) $\sigma_n$ has a
    convergent subsequence, (2) a subsequence along which $\sigma_n \to +\infty$ or (3) a
    subsequence along which $\sigma_n \to -\infty$.
    Case 1: After passing to a subsequence, $\sigma_n \to \sigma_0$. Then,
    after passing to a further subsequence as necessary, we also have $u_n \to
    u_0$ in $C^\infty_\loc$, where $u_0$ satisfies the continuation map
    equation \eqref{eqn:continuation map pde}, but interpolating from $F_0$ to
    $H_0$.

    By the energy estimate in Equation \eqref{eqn:Energy of u_n}, $u_0$ must have zero energy, and hence have image in a shared (constant) chord of $F_0, H_0$. 
    On the other hand, since $\sigma_n \to \sigma_0$, $H( u_0(\sigma_0, 0) ) =
    \frac{C}{2}$, which gives a contradiction.

    Case 2: After passing to a subsequence, $\sigma_n \to +\infty$. We
    consider then the reparametrized strips $v_n \colon \R \times [0,1] \to W$ by
    $v_n(s,t) = u_n(s+\sigma_n, t)$.  After passing to a further subsequence, we 
    may have that $v_n$ converges in $C^\infty_\loc$ to a strip $v_0$ with boundary
    on $L$, but satisfying the equation
    \[
        \partial_s v_0 + J \partial_t v_0 + \nabla H_0(v_0) = 0.
    \]
    The energy of the limit $v_0$ is bounded above by the limit of the energy from
    Equation \eqref{eqn:Energy of u_n}. Thus, $v_0$ has zero energy, 
    and hence has its image in a chord of $H_0$, and thus its image is constant. However,
    we have $H(v_n(0,0)) = \lim H(u_n(\sigma_n, 0)) = \frac{C}{2}$. This gives
    a contradiction.

    Case 3: If instead, $\sigma_n \to -\infty$, the same argument as in Case 2
    gives a contradiction, but instead using $F_0$ instead of $H_0$.
 
    These three cases therefore show that for $n$ sufficiently large, 
    and for $0< a< C < b < a_0$, 
    \begin{align*}
        \Phi_n \colon CF^{<b}(F_n; L) &\to  CF^{<b}(H_n; L)\\
        \Phi_n \colon CF^{<a}(F_n; L) &\to  CF^{<a}(H_n; L)
    \end{align*}
    In particular then, $\Phi_n$ induces a map 
    \[
        \Phi_n \colon HF^{(a,b)}(F_n) \to HF^{(a,b)}(H_n)
    \]
    For $n$ sufficiently large, $\Phi_n$ induces the map $\Phi$ (as in
    Definition \ref{def:degenerate filtered Floer})
    \[
        \Phi \colon HF^{(a,b)}(F_0) \to HF^{(a,b)}(H_0).
    \]

    A similar argument (reversing the roles of $H_0, F_0$) 
    allows us to show that, for $n$ large, and for $0< a< C < b < a_0$, 
    \begin{align*}
        \Psi_n \colon CF^{<b}(H_n; L) &\to  CF^{<b}(F_n; L)\\
        \Psi_n \colon CF^{<a}(H_n; L) &\to  CF^{<a}(F_n; L)
    \end{align*}
    and hence, $\Psi_n$ induces a map 
    \[
        \Psi_n \colon HF^{(a,b)}(H_n) \to HF^{(a,b)}(F_n).
    \]
    As before, for $n$ large, $\Psi_n$ induces the map
    \[
        \Psi \colon HF^{(a,b)}(H_0) \to HF^{(a,b)}(F_0).
    \]

    Finally, $\Psi_n$ and $\Phi_n$ (as chain maps) are chain inverses (for the large $n$ for
    which they are defined), and hence $\Phi, \Psi$ are isomorphisms on
    homology.
\end{proof}

\begin{proof}[Proof of Theorem \ref{thm:filtered homology of the maximum}]

    We first apply Lemma \ref{lem:exists small admissible representative} to
    replace $H_0$ by an admissible Hamiltonian $H_1$ with a small maximum $\alpha$, where
    $0 < \alpha < \frac{a_0}{3}$. 
    tony{Now that I think about it, I don't understand why this rescaling is necessary. Can we not prove directly that the $HF(H_0) \cong \Z/2[n]$}
    Furthermore, we choose $\alpha$ sufficiently small
    and $dH_1$ sufficiently small that the hypotheses of Lemma 
    \ref{lem:Floer homology from Morse} are satisfied.

    Then (by Lemma \ref{lem:exists small admissible representative})
    for $0 < \epsilon < \frac{1}{2} \alpha$ we have 
    \[
        HF^{(m(H_0)-\epsilon, m(H_0)+\epsilon)}(H_0) \cong 
        HF^{(m(H_1)-\epsilon, m(H_1)+\epsilon)}(H_1) 
    \]

    Now, by Lemma \ref{lem:Floer homology from Morse}, we have 
    \[
        HF^{(m(H_1)-\epsilon, m(H_1)+\epsilon)}(H_1) 
        \cong
        HF^{(m(F_1)-\epsilon, m(F_1)+\epsilon)}(F_1) 
    \]
    where $F_1 = \widetilde{H_1|_{L}}$ as in Lemma \ref{lem:admissible from
    function}.

    Finally, from Lemma \ref{lem:admissible from function}, we have
    \[
        HF^{(m(F_1)-\epsilon, m(H_1)+\epsilon)}(F_1)  \cong \Z_2[n].
    \]

    From all of this, we therefore obtain
    \[
        HF^{(m(H_0)-\epsilon, m(H_0)+\epsilon)}(H_0) \cong \Z_2[n].
    \]

    Write $\kappa = \floor{\frac{m(H_0)}{a_0}}$.
    The result now follows from the fact that \[
    (a,b) \cap \Sigma(H_0) = \{ m(H_0) \}\]
    for all $a, b$ with  $\kappa a_0 < a < m(H_0) < b < (\kappa+1)a_0$, and in particular, for $a = m(H_0) - \delta$, 
    $b = m(H_0) + \delta$.
\end{proof}

\section{Proof of Theorem \ref{T:displacement}}
\label{S:displaceableLag}

Let $H$ be an admissible Hamiltonian in 
the sense of Definition \ref{def:modified_admissible}, and 
as before, we write $a_0$ for the area of a disk with minimal Maslov
number $N_L$. Fix a Morse function $f \colon L \to \R$, and use it to define the Quantum
Lagrangian complex $CQ(L)$. Let $m_0$ be the class of the maximum in $CQ(L)$.
Since $L$ is displaceable, there exists $y \in CQ(L)$ so that $\partial y = m_0$
and $\A_0(y) \le d(L)$, where $d(L)$ is the displacement energy of $L$.
Notice also that $\A_0(y) = \kappa_0 a_0$ for some $\kappa_0 \in \Z$.

The main idea of this proof is to consider the homotopy from $0$ (whose Floer
homology can be considered to be $QL$) to $H_0$, and the corresponding induced maps on the filtered
Floer homology groups. In order to control the continuation maps, we will need
to consider the composition of the continuation maps as we ``slice up'' the
homotopy.
Let $0 = \tau_0 < \tau_1 < \tau_2 < \dots < \tau_N = 1$.
Then, we will consider the following Hamiltonians:
\[
    0 \le \tau_1 H_0 \le \tau_2 H_0 \le \dots \le \tau_N H_0 = H_0
\]
and the continuation maps from $\tau_k H_0$ to $\tau_{k+1} H_0$.

We consider first the case of the map from $0$ to $\tau_1 H_0$, and afterwards
will consider the maps corresponding to $k \ge 1$.

\begin{lem} \label{lem:base case}

    Let $H_1$ be an admissible Hamiltonian with $0 < m(H_1) < \frac{a_0}{2}$.

    Let $\delta > 0$ be sufficiently small that \begin{align*} 
        &(m(H_1) - 2\delta, m(H_1) + 2\delta) \cap \Sigma(H_1) = \{ m(H_1) \}.  
    \end{align*}

    Let $m_0$ denote the class of the maximum in $CQ(L)$, and let $y \in CQ(L)$ 
    so that $\partial y = m_0$, with $\A_0(y) = \kappa_0 a_0$ for $\kappa_0 \in
    \Z^+$.

   Then the map 
   \[
       \imath \colon HF^{(m(H_1) - \delta, m(H_1) + \delta)}(H_1) \to
       HF^{(m(H_1) - \delta, \kappa_0 a_0+\delta)}(H_1)
   \]
   is the zero map.
\end{lem}
\begin{proof}

    Let $H$ be non-degenerate and $\epsilon$ $C^2$-close to $H_1$. By making
    $\epsilon > 0$ sufficiently small, we have (by definition):
    \begin{align*}
       HF^{(m(H_1) - \delta, m(H_1) + \delta)}(H_1) 
       &= HF^{(m(H_1) - \delta, m(H_1) + \delta)}(H)\\ 
       HF^{(m(H_1) - \delta, k_0 a_0+\delta)}(H_1)
       &= HF^{(m(H_1) - \delta, k_0 a_0+\delta)}(H)
   \end{align*}

   By Theorem \ref{thm:filtered homology of the maximum}, we have 
   $HF^{(m(H_1) - \delta, m(H_1) + \delta)}(H) \cong \Z_2[n]$.

   Notice also that the action spectrum of $H$ is contained in the following
   \begin{equation} \label{eqn:action spectrum H}
       \Sigma(H) \subset \bigcup_{k \in \Z} \bigg ( (-\epsilon + k a_0, \epsilon + ka_0) \cup 
                    (-\epsilon + m(H_1) + k a_0, \epsilon + m(H_1) + ka_0 )
                    \bigg ).
    \end{equation}

   Consider now the continuation map 
   \begin{align*}
       \Psi &\colon CQ(L) \to CF(H) 
   \end{align*}
   This map $\Psi$ increases action by at most $m(H_1)+\epsilon$. 
   Furthermore, by Theorem \ref{thm:filtered homology of the maximum}, $\Psi(m_0)$ is
   the generator of 
   \begin{equation}\label{eqn:max class non-zero}
		 HF^{(m(H_1) - \delta, m(H_1) + \delta)}(H) \cong \Z_2[n].
   \end{equation}

   Furthermore, the image $\Psi(y)$ has action at most $m(H_1)+\epsilon + \kappa_0
   a_0$ by \ref{eqn:action estimate}. 
    Therefore, the class of $\Psi(m_0)$ is $0$ in $HF^{(m(H_1) - \delta, m(H_1) +
    \epsilon + \kappa_0 a_0)}(H)$. However, by Theorem \ref{thm:filtered homology of the maximum}
   \begin{equation*}
   	HF^{I(m(H_1),\epsilon)}(H) \cong \Z/2[n]
   \end{equation*} and after multiplication by $t^{-\kappa_0}$, Equation \ref{eqn:t multiplication} gives
   \[
   HF^{(m(H_1) + \kappa_0 a_0 -\epsilon, m(H_1)+\kappa_0 a_0 + \epsilon)}(H)  \cong
   \Z/2[n+\kappa_0 N_L]
   \]
   Since $N_L \ge 2$ and $|y| = n+1$, it follows that $\Psi(y)\notin HF^{I(m(H_1)+\kappa_0 a_0,\epsilon)}(H_1)$. Since, by hypothesis, $m(H_1) < a_0/2$, there are no generators of the complex in the action window $(\kappa_0 a_0 + \epsilon, m(H_1) + \kappa_0 a_0 - \epsilon)$, and therefore $\A(\Psi(y)) < \kappa_0 a_0 +\epsilon$
   and $\Psi(y) \in HF^{(m(H_1)-\delta,\kappa_0 a_0 + \delta)}(H)$, as desired.
\end{proof}

\begin{lem} \label{lem:inductive step}

    Let $H_1, H_2$ be admissible Hamiltonians such that 
    $0 < m(H_2) - m(H_1) < \frac{a_0}{2}$ 
    and such that $m(H_j) \notin \Z a_0$.

    Let $\delta > 0$ be sufficiently small that \begin{align*} 
        &(m(H_j) - 2\delta, m(H_j) + 2\delta) \cap \Sigma(H_j) = \{ m(H_j) \}  
    \end{align*}
    for $j=1, 2$.

    Suppose furthermore that \[
        \imath \colon HF^{(m(H_1) - \delta, m(H_1)+\delta)}(H_1) \to 
    HF^{(m(H_1) - \delta, k_0 a_0 +\delta)}(H_1) 
   \]
   is the zero map for some $k_0 \in \Z$ with $m(H_1) < k_0 a_0$.

   Then, $m(H_2) < k_0 a_0$ and the map 
   \[
       \imath \colon HF^{(m(H_2) - \delta, m(H_2) + \delta)}(H_2) \to
       HF^{(m(H_2) - \delta, k_0 a_0+\delta)}(H_2)
   \]
   is the zero map.

\end{lem}

\begin{proof}

    Let $\Delta = m(H_2) - m(H_1) < \frac{a_0}{2}$. 

    By considering continuation maps, which we denote by $\Psi$, we have
    \[
        \Psi \colon HF^{(a,b)}(H_1) \to HF^{(a+\Delta, b+\Delta)}(H_2) 
    \]
    for each interval $(a,b)$ for which these both of these filtered homology groups are defined.

    From Theorem \ref{thm:filtered homology of the maximum}, we have that the map
    \[
        \Psi \colon HF^{(m(H_1)-\delta, m(H_1) + \delta)}(H_1) \to
        HF^{(m(H_2)-\delta, m(H_2) + \delta)}(H_2)
    \]
    is an isomorphism (and both of these are isomorphic to $\Z/2$, in degree
    $n$).

    Combining these observations with the naturality of these maps, the following diagram commutes:
    \begin{equation*}
       \begin{tikzcd}
           HF^{(m(H_1)-\delta, m(H_1)+\delta)}(H_1) \arrow{r}{0} \arrow[swap]{d}{\Psi}
           &  HF^{(m(H_1)-\delta, k_0 a_0 +\delta)}(H_1) \arrow{d}{\Psi} \\
           HF^{(m(H_2)-\delta, m(H_2)+\delta)}(H_2) \arrow[swap]{r}{\imath} 
           &  HF^{(m(H_2)-\delta, k_0 a_0 + \Delta + \delta)}(H_2) 
  \end{tikzcd} 
    \end{equation*}
    and hence 
    \begin{equation}\label{eqn:iota zero map}
        \imath \colon 
           HF^{(m(H_2)-\delta, m(H_2)+\delta)}(H_2) \xlongrightarrow{\imath} 
           HF^{(m(H_2)-\delta, k_0 a_0 +\Delta + \delta)}(H_2)
    \end{equation}
    is the zero map.

    We now have two possibilities: either $m(H_2) > k_0 a_0$ or $m(H_2) < k_0
    a_0$.

    Consider first the case when $m(H_2) > k_0 a_0$. 
    Then, we observe that 
    $[m(H_2)+\delta, k_0 a_0 + \Delta + \delta ] \cap \Sigma(H_2) = \emptyset$
    since $k_0 a_0 + \Delta + \delta < (k_0+1) a_0 < m(H_2) + a_0$.
    Thus, by Equation \eqref{eqn:shrink interval}, 
     \[
        \imath \colon 
           HF^{(m(H_2)-\delta, m(H_2)+\delta)}(H_2) \longrightarrow 
           HF^{(m(H_2)-\delta, k_0 a_0 +\Delta + \delta)}(H_2)
    \]
    is an isomorphism, contradicting that $\iota$ is the zero map. 
    
    Hence, it follows that we must be in the second case that $m(H_2) < k_0 a_0$. Consider the exact triangle (of the r2p structure on $HF^{(a,b)}(H_2)$)
    associated to $m(H_2) - \delta < k_0 a_0 + \delta < k_0 a_0 + \Delta +
    \delta$ (as in Equation \ref{eqn:exact triangle}).

    This gives us the following:
    \begin{equation}\label{eqn:Exact triangle sec 4}
        \begin{split}
            HF^{(m(H_2) - \delta, k_0 a_0 +\delta)}&(H_2) 
        \xlongrightarrow{i_1} 
        HF^{(m(H_2) - \delta, k_0 a_0 + \Delta + \delta)}(H_2) 
        \xlongrightarrow{i_2} \\
            \xlongrightarrow{i_2} 
        &HF^{(k_0 a_0 + \delta, k_0 a_0 + \Delta + \delta)}(H_2)
        \xlongrightarrow{\delta} 
        HF^{(m(H_2) - \delta, k_0 a_0 +\delta)}(H_2)[1]
    \end{split}
    \end{equation}

    Consider the intersection 
    \[
        [k_0 a_0 + \delta, k_0 a_0 + \Delta + \delta] \cap \Sigma(H_2).
    \]
    There are two possibilities: this intersection can be trivial, in which
    case 
        \[
         \imath \colon
        HF^{(m(H_2) - \delta, k_0 a_0 +\delta)}(H_2) \longrightarrow
         HF^{(k_0 a_0 + \delta, k_0 a_0 + \Delta + \delta)}(H_2) = {0}
        \]
        is an isomorphism, and hence the claimed result follows. Alternately, it may be that for some $l \in \Z$, 
    \[
        [k_0 a_0 + \delta, k_0 a_0 + \Delta + \delta] \cap \Sigma(H_2) = \{
        m(H_2) + l a_0 \}.
    \]
    In this case, by Equation \eqref{eqn:shrink interval}, and for $\epsilon >
    0$ sufficiently small, we obtain that 
        \[
            HF^{(k_0 a_0 + \delta, k_0 a_0 + \Delta + \delta)}(H_2) \cong 
            HF^{(m(H_2) + la_0 -\epsilon, m(H_2)+la_0 + \epsilon)}(H_2) 
        \]
        By Theorem \ref{thm:filtered homology of the maximum} 
        and Equation \eqref{eqn:t multiplication}, it follows that 
        \[
            HF^{(m(H_2) + la_0 -\epsilon, m(H_2)+la_0 + \epsilon)}  \cong
            \Z/2[n+l N_L]
        \]
        where $\Z/2[n+l N_L]$ is $\Z/2$ in degree $n+l N_L$.

		Let $\eta\colon HF^{(m(H_2)-\delta,m(H_2)+\delta)}(H_2) \to HF^{(m(H_2)-\delta,k_0 a_0+\Delta + \delta)}(H_2)$ be the inclusion. Note that
		\[\iota = i_1 \circ \eta\colon 
		HF^{(m(H_2)-\delta, m(H_2)+\delta)}(H_2) \longrightarrow 
		HF^{(m(H_2)-\delta, k_0 a_0 +\Delta + \delta)}(H_2)\]
		is the zero map by Equation \ref{eqn:iota zero map}, so that $\text{Im}(\eta) \subset \ker(i_1)$. Suppose that $\text{Im}(\eta) \neq \{0\}$. By the exactness of the sequence \ref{eqn:Exact triangle sec 4}, $\text{Im}(\eta) \subset \text{Im}(\delta)$. However, every element in the image of $\delta$ has degree $n+lN_L -1$, and every element in  $HF^{(m(H_2)-\delta, m(H_2)+\delta)}(H_2)$ has degree $n$, which is a contradiction, and therefore $\text{Im}(\eta)= \{0\}$. This proves the result.
\end{proof}

\begin{proof}[Proof of Theorem \ref{T:displacement}]
Theorem \ref{T:displacement} follows from the above lemmas by considering
$0 = \tau_0 < \tau_1 < \dots < \tau_n = 1$ and the corresponding Hamiltonians
$0=\tau_0 H_0, \tau_1 H_0, \tau_2 H_0, \dots, \tau_n H_0 = H_0$, where we choose
$\tau_{k+1} - \tau_k$ sufficiently small so that $(\tau_{k+1}-\tau_k)m(H_0) <
\frac{a_0}{2}$ and $\tau_k m(H_0) \notin \Z a_0$. 
We first apply Lemma \ref{lem:base case} to $\tau_1 H_0$. For each pair 
$\tau_k H_0$ and $\tau_{k+1}H_0$ for $k \ge 1$, we then apply Lemma
\ref{lem:inductive step}.

From this, we conclude from \ref{lem:inductive step} that $m(\tau_k H_0) < \kappa_0 a_0 \le d(L)$ for each
$k$, and, in particular, that $m(H_0) < \kappa_0 a_0 \le d(L)$.

\end{proof}

\bibliography{Mathematics}
\end{document}